\documentclass{amsart}
\usepackage{amsfonts}
\usepackage{latexsym}
\usepackage{amssymb}
\usepackage{amsmath}

\newcommand{\R}{\mathbb R}

\newtheorem{thm}{Theorem}[section]
\newtheorem{cor}{Corollary}[section]

\newtheorem{proposition}{Proposition}[section]
\theoremstyle{remark}
\newtheorem*{rmk}{Remark}

\begin{document}


\bigskip
\title{Factoring Sobolev inequalities through classes of functions}

\author[D.\,Alonso]{David Alonso-Guti\'errez}
\address{Departamento de Matem\'aticas, Universidad de
Zaragoza, 50009 Zaragoza, Spain}\email[(David
Alonso)]{daalonso@unizar.es}
\thanks{ The three  authors partially supported by MCYT
Grant(Spain) MTM2007-61446 and DGA E-64}

\author[J.\,Bastero]{Jes\'us Bastero}\email[(Jes\'us
Bastero)]{bastero@unizar.es}

\author[J.\,Bernu\'es]{Julio Bernu\'es}
\email[(Julio Bernu\'es)]{bernues@unizar.es}

\begin{abstract} We recall two approaches to
recent improvements of the classical Sobolev inequality. The first one follows the point of view of Real Analysis, \cite{MM1}, \cite{BMR}, 
while the second one relies on tools from Convex Geometry, \cite{Z}, \cite{LYZI}.  In this paper we prove a (sharp) connection between them. 

\end{abstract}

\date{}

\maketitle

\section{Introduction and notation}

 The classical Sobolev inequality states that for $1\le p<n$ and $\frac1q=\frac1p-\frac1n$, there exists a constant
$C_{p,n}>0$ such that for every  $f\colon\R^n\to\R$ in the Sobolev space $W^{1,p}(\R^n)$,
\begin{equation}\label{Sobolev}
 \Vert\nabla f\Vert_p\ge C_{p,n} \Vert f\Vert_q
\end{equation}
where $\Vert \cdot\Vert_q$ denotes the $L_q-$norm of the Euclidean norm of
functions and $\nabla f$ is the gradient of $f$.

The best constant in the case $p=1$ ($q=n/(n-1)$) was
obtained by  H. Federer, W. Fleming, \cite{FF} and independently by V. Maz'ja, \cite{M1} \cite{M2}. They proved 
$C_{1,n}=n\ \omega_n^{\frac{1}{n}}$, where $\omega_n$  denotes the Lebesgue measure of the
Euclidean unit ball in ${\mathbb R}^n$, and showed that this fact is equivalent to the isoperimetric
 inequality (see for instance \cite{C} for a survey). 
For the other values of $1<p<n$, Aubin and Talenti got the best constants, \cite{A}, \cite{T1}. 
See also the recent approaches in \cite{BL} and \cite{CNV}.

We point out that one key step in classical proofs of (\ref{Sobolev}) is the use of
Polya-Szeg\"o rearrangement inequality, see \cite{PS},
\begin{equation}\label{PS}
\Vert \nabla f\Vert_p \ge \Vert \nabla f^\circ\Vert_p,\qquad \qquad p\ge 1
\end{equation}
where $f^{\circ}(x):=f^{*}(\omega_n |x|^n)$, is a radial extension to ${\mathbb R}^n$  of the decreasing
rearrangement of $f$, $f^{*}(t)=\inf\{\lambda>0\ ;\ |\{|f|>\lambda\}|_n\le t\}, t\ge 0$,  $|\cdot |$ is the 
Euclidean distance in ${\mathbb R}^n$ and $|\cdot|_n$ is the Lebesgue measure on the (suitable) $n$-dimensional
space. $f^{\circ}$ has the same distribution function as $f$ and $f^{\ast}$. It 
is called the symmetric Schwarz rearrangement of $f$.

For $p=n$, the inequality with $q=\infty$ is not true. 
In the sixties Trudinger \cite{Tr} and Moser \cite{M} proposed an Orlicz space, $\mathcal{MT}(\Omega)$, 
of functions defined on open domains $\Omega\subset {\mathbb R}^n$ with $|\Omega|_n<\infty$ and showed the continuous
inclusion $W^{1,n}_0(\Omega)\hookrightarrow\mathcal{MT}(\Omega)$, where $W^{1,n}_0(\Omega)$ is the closure of the space of 
 $\mathcal{C}^1$ functions of compact support,
$\mathcal{C}^1_{00}(\Omega)$, in the Sobolev space $W^{1,n}(\Omega)$. More precisely, they proved that 
there exists $C_n>0$ such that for all $f\in W^{1,n}_0(\Omega)$
\[
\Vert \nabla f\Vert_n \geq C_n \Vert f\Vert_{\mathcal{MT}}
\]
and the constant (depending on $|\Omega|_n$) is sharp.

In the late seventies, Hansson \cite{Ha} and Brezis-Wainger \cite{BW} improved the target space in the inclusion above.
They introduced a rearrangement invariant function space, $H_n(\Omega)$,
such that $W_0^{1,n}(\Omega)\hookrightarrow H_n(\Omega)\hookrightarrow\mathcal{MT}(\Omega)$. 
Moreover, $H_n(\Omega)$ was proved to be the optimal target space in the class of rearrangement invariant spaces. 
Equivalently, they obtained an inequality of the form
$\Vert \nabla f\Vert_n\geq c_n \Vert f\Vert_{H_n}\geq c'_n \Vert f\Vert_{\mathcal{MT}}$ 
for some constants $c_n, c'_n>0$ (depending on $|\Omega|_n$). 

Tartar \cite{Tar}, Maly-Pick \cite{MP} and Bastero-Milman-Ruiz \cite{BMR}, see also \cite{K},
refined those estimates using {\sl classes of functions} as follows: For $1\le
p<\infty$ denote
 \[\mathcal{A}_{\infty,p}(\R^n)=\{f; \Vert f\Vert_{\infty,
p}=\left(\int_0^{\infty}
(f^{\ast\ast}(t)-f^\ast(t))^p\frac{dt}{t^{p/n}}\right)^{1/p}<\infty\}\]
where
 $f^{\ast\ast}$ is the Hardy transform of  $f^{*}$ defined by
$\displaystyle f^{\ast\ast}(t)=\frac{1}{t}\int_0^tf^\ast(s)\ ds.$$W^{1,p}(\R^n)$
Then for all $f\in W^{1,n}_0(\Omega)$ 
\begin{equation}\label{BMR} \Vert \nabla f\Vert_n\geq
 (n-1)\,\omega_n^{\frac1n}\Vert f\Vert_{\infty,n}\geq c''_n \Vert
f\Vert_{H_n}
\end{equation}
for some $c''_n>0$ (depending on $|\Omega|_n$). Observe that the constant in the first inequality 
 depends neither on the measure of $\Omega$ nor on the support of $f$.

Once one considered classes of functions instead of vector spaces, Sobolev type
inequalites could be extended further. At this point we recall the well known 
fact that $W_0^{1,p}(\R^n)=W^{1,p}(\R^n)$, \cite{M2}.  In \cite{MM1} the authors proved
\begin{equation}
\label{des}\Vert \nabla f\Vert_p\geq c_{n,p}
 \Vert f\Vert_{\infty,p}\geq c'_{n,p} \Vert
f\Vert_q,\qquad  \forall\ f\in W^{1,p}(\R^n),\qquad 1\le p<n
\end{equation}
and some constants $ c_{n,p}, c'_{n,p}>0$. 

We now move on to a different philosophy. We start by recalling the so
called Petty projection inequality, stated in \cite{P} for
convex bodies and extended by Zhang \cite {Z} to compact subsets $K\subset{\mathbb R}^n$ 
\begin{equation}\label{Petty}
 \frac{n\omega_n}{\omega_{n-1}}\left(\int_{
S^{n-1}}
\frac{du}{|P_{u^{\perp}}(K)|_{n-1}^n}
\right)^{-\frac{1}{n}}\ge n\omega_n^{1/n}
 |K|_n^{\frac{n-1}{n}}
\end{equation}
 where $P_{u^{\perp}}$ is the the orthogonal projection onto the hyperplane
$u^{\perp}$ and $du$ is the normalized Haar
probability on the unit sphere $S^{n-1}$. Petty projection inequality directly implies the isoperimetric inequality.

In 1999 Zhang \cite {Z} (see also \cite{HSX} and the references therein)
introduced a new class of functions
\[\mathcal{E}_p(\R^n)=\left\{f\in W^{1,p}(\R^n);
\mathcal{E}_p(f):= \frac1{I_p}\left(
\int_{S^{n-1}}\Vert D_u f\Vert_p^{-n}du\right)^{-\frac{1}{n}}<\infty\right\}, \qquad p\ge 1
\]
where $D_uf(x):=\langle\nabla f(x), u\rangle $ and $\displaystyle
I^{p}_p:=\int_{S^{n-1}}|u_1|^pdu$ is a normalization constant so that $\mathcal{E}_p(f^{\circ})= \Vert \nabla
f^{\circ}\Vert_p$. The expression $\mathcal{E}_p(f)$ is an energy integral having applications in
information theory. It is invariant under transformations of $\R^n$ of the form $x\to x_0+Ax, x_0\in\R^n, A\in SL(n)$, 
\cite{LYZII}.
Moreover, by Jensen's inequality and Fubini's theorem the following relation holds
$$\mathcal{E}_p(f)=\frac{1}{I_p}
\left(\int_{S^{n-1}}\Vert D_uf\Vert_p^{-n}du\right)^{-1/n}
\le \frac{1}{I_p}\left(\int_{S^{n-1}}\Vert D_uf\Vert_p^{p}du\right)^{1/p} =\Vert \nabla f\Vert_p.$$
 
The following remarkable inequality
\begin{equation}\label{LYZ}
\mathcal{E}_p(f) \geq
\mathcal{E}_p(f^\circ)\ ,\qquad  1\le p<\infty
\end{equation}
was proved in a series of papers: Zhang \cite {Z} initiated the approach 
by showing that his extension of the Petty projection 
inequality (\ref{Petty}) is actually equivalent to (\ref{LYZ}) for $p=1$. The general case was proved via the
$L_p$-Brunn-Minkowski theory in \cite{LYZI}, \cite{LYZII}, \cite{CLYZ}. The invariance of $\mathcal{E}_p(f)$ implies, by homogeneity, that 
(\ref{LYZ}) is affine-invariante i.e. invariant under transformations of $\R^n$ of the form $x\to x_0+Ax, x_0\in\R^n, A\in GL(n)$.

The inequality (\ref{LYZ}) is stronger than Polya-Szeg\"o
rearrangement inequality (\ref{PS}) and thus 
it yields to a new proof of Sobolev's inequality  
\begin{equation}\label{sobolev}
 \Vert\nabla f\Vert_p\ge \mathcal{E}_p(f) \geq
\mathcal{E}_p(f^\circ) \ge C_{n,p}\Vert f\Vert_{q}, \qquad 1\le
p<n.
\end{equation}
See \cite {Z}, \cite{LYZII} for such a proof of (\ref{Sobolev}) with sharp constants. 
We remark the fact that $\mathcal{E}_p(f) \geq C_{n,p}\Vert f\Vert_{q}$ 
is affine-invariant while Sobolev's inequality $\Vert\nabla f\Vert_p\ge C_{n,p}\Vert f\Vert_{q}$ is 
invariant only under rigid motions. 

In  \cite{HS1} \cite{HS2}, \cite{HSX}, the authors
investigated the space $\mathcal{E}_p^+(\R^n)$ defined analogously as before by
\[
 \mathcal{E}_p^+(f):=\frac{2^{1/p}}{I_p} \left(
\int_{S^{n-1}}\Vert D_u^+ f\Vert_p^{-n}du\right)^{-\frac{1}{n}}
\]
where $D_u^+f(x):=\max\{\langle\nabla f(x), u\rangle,0 \} $ and proved
\begin{equation}\label{HS}
 \mathcal{E}_p(f)\geq \mathcal{E}_p^+(f) \geq
\mathcal{E}_p^+(f^\circ)=\mathcal{E}_p(f^\circ), \qquad p\ge 1
\end{equation}
which refined (\ref{sobolev}) for $1\le
p<n$.

In the case $p\ge n$, affine-invariant inequalities of
Sobolev type were studied in \cite{CLYZ}, \cite{HS2} and \cite{HSX} with the
hypothesis of $f$ having support of finite measure. In the limiting case $p=n$ they proved the sharp inequality
\begin{equation}\label{hs}
 \Vert \nabla f\Vert_n
\geq \,\mathcal
{E}_n(f) \geq \mathcal
{E}_n^+(f) \geq C_n \Vert f\Vert_{\mathcal{MT}}
\end{equation}
while for $p>n$,
\begin{equation}\label{hsp}
 \mathcal
{E}_p(f) \geq\mathcal
{E}_p^+(f)  \geq \left(\frac{p'}{|q|}\right)^{\frac{1}{p'}} n\omega_n^{1/n}\Vert f\Vert_\infty |{\text {supp}\,f}|_n^{1/q}
\qquad \text{where}\quad \frac 1p +\frac 1{p^\prime}=1
\end{equation} where the constant {\sl depending on the size of the support of $f$} is sharp 
(take $f^{\ast}(t)=(1-t^{-p'/q})\chi_{[0,1]}\!(t)$).

In conclusion, the first approach so described looked for improvements of the right hand side of (\ref{Sobolev}), while
the second approach concerned the left hand side of (\ref{Sobolev}). In this note we link these two approaches
and show that for all $1\le p <\infty$ and $\displaystyle \frac
1q:=\frac1p-\frac1n$
\[\mathcal{E}_p^+(f)
\geq\left(1-\frac 1q\right){n\
\omega_n^{1/n}  }  \Vert
f\Vert_{{\infty,p}}\qquad \forall f\in W^{1,p}(\R^n)\]
(see Theorem \ref{relation}) where the constant is sharp.
As a consequence, for $1\le p<n$ it gives the right constant in the first inequality in 
(\ref{des}) and enables to connect (\ref{sobolev}), (\ref{HS}) and (\ref{des}). For $p=n$, 
Theorem \ref{relation} and its Corollary \ref{cor} connect (\ref{BMR}) and (\ref{hs}), improving 
the first inequality in (\ref{BMR}). In the case $p>n$ it links the first inequality 
in (\ref{des}) and (\ref{sobolev}), (thanks to (\ref{HS}) and to the fact that 
they are also valid for $p>n$). Moreover, in Proposition \ref{mejora} we see how 
it yields to lower estimates for $\mathcal{E}_p^+(f)$ better than (\ref{hsp}).

In the third section we include a proof of the
inequalities (\ref{LYZ})  and (\ref{HS}) which directly derives from Zhang's extension of Petty projection
inequality, paying the penalty of an extra constant $\frac{I_p}{I_1}$ (which is independent of the dimension $n$).
No use of the $L_p$-Brunn-Minkoswki theory and polytope
approximation appearing in the papers \cite{LYZI}, \cite{LYZII}, \cite{CLYZ} and
\cite{HS2} is made.

\section{The results}

The first result is the correct relation between $\mathcal {E}_p^+(f)$ and $\Vert f\Vert_{\infty,p}$.

\begin{thm}\label{relation} Let $1\le p<\infty$ and $\displaystyle \frac
1q=\frac1p-\frac1n$, $q\in (-\infty, -n)\cup[\displaystyle
\frac n{n-1}, \infty]$. Then
\[\mathcal
{E}_p^+(f)
\geq\left(1-\frac 1q\right){n\
\omega_n^{1/n}  }  \Vert
f\Vert_{{\infty,p}}\qquad \forall f\in W^{1,p}(\R^n).
\]           Moreover the constant is sharp.
\end{thm}

\begin{proof}
Taking (\ref{HS}) into account it is
enough to see that for any
 $f:\R^n\to \R$ compactly supported $C^{1}$
function 
\[
\mathcal{E}_p^+(f^{\circ})=\frac{2^{1/p}}{I_p}  \left(\int_{S^{n-1}}\Vert D_u^+
f^\circ\Vert_p^{-n}du \right)^{-1/n}\geq \left(1-\frac 1q\right){n\
\omega_n^{1/n} }  \Vert f\Vert_{\infty,p}.
\]

Now, $f^\circ$ is Lipschitz and $f^\ast$ locally Lipschitz
and $f^{\ast\prime}$ integrable (see for instance \cite{K} and \cite{Mo}). Therefore since
$$ f^{\ast\ast}(t)-f^{\ast}(t)=\frac1t\int_0^t(f^\ast(s)-f^\ast(t))ds=\frac 1t
\int_0^t\int_s^t- f^{\ast\prime}(u)duds\\
=\frac 1t \int_0^ts
|f^{\ast\prime}(s)|ds
$$ we have
$$
 \Vert f^\circ\Vert_{\infty, p}^p =\Vert f\Vert_{\infty, p}^p=\int_0^\infty
(f^{\ast\ast}(t)-f^{\ast}(t))^p\frac{dt}{t^{p/n}}
=\int_0^\infty\left(\frac 1t \int_0^ts
|f^{\ast\prime}(s)|ds
\right)^p\frac{dt}{t^{p/n}}.$$
Apply Hardy's inequality $\int_0^\infty\left(\frac 1t \int_0^tg(s)ds
\right)^p\frac{dt}{t^{p/n}}\le\left( \frac{p}{p+\frac{p}{n}-1}\right)^p\int_0^\infty g(s)^p 
\frac{ds}{s^{p/n}}$ to $g(s)=s |f^{\ast\prime}(s)|$ (see \cite{M2} for a reference on Hardy's inequalities with 
weights) to obtain
$$\Vert f^\circ\Vert_{\infty, p}^p\le \left(\frac{1}{1-\frac1q}\right)^p\int_0^\infty s^{\frac{(n-1)p}{n}}|f^{\ast\prime}(s)|^p\
ds.$$

On the other hand, by definition of $f^\circ(x)$,
$$
\langle \nabla f^\circ(x),u\rangle_+=n\omega_n|x|^{n-1}|f^{\ast\prime}(\omega_n
|x|^n)|\left\langle\frac{-x}{|x|},u\right \rangle_+
$$
 and
so by polar integration $x=r\theta$
 and the change of variables $s=\omega_n r^n$
\begin{equation}
 \label{fpolar}
 \Vert D_u^+f^\circ\Vert_p^p=\int_{\R^n}\langle\nabla
f^\circ(x),u\rangle_+^p\, dx=\frac12I_p^p
\Big(n\omega_n^{1/n}\Big)^p\int_0^\infty
s^{\frac{(n-1)p}{n}}|f^{\ast\prime}(s)|^pds
\end{equation}
 and  the result follows. 

In order to see that the constant is sharp consider truncations
of the function
$f^\ast(t)=t^{-1/q}$, whenever $p<n$, $f^\ast(t)=\log(1/t)$, for $p=n$ and
$f^\ast(t)=(1-t^{-1/q})\chi_{[0,1]}$, whenever $p>n$.
\end{proof}

Straightforward computations show that $f^{\ast}$ and therefore
$\Vert f\Vert_{\infty, p}$ are invariant under transformations of $\R^n$, $x\to x_0+Ax, x_0\in\R^n, A\in SL(n)$. Consequently, 
the inequality in Theorem \ref{relation} is affine-invariant. 

\begin{cor}\label{cor}
 Let $1\leq p<\infty$. For any  $f\in W^{1,p}(\R^n)$
\[\Vert \nabla f\Vert_p\ge
n\ \omega_n^{1/n} \left(1-\frac 1q\right)\Vert f\Vert_{\infty, p}.
\]
In  particular, for $p=n$ we have $\Vert \nabla f\Vert_n\ge n\ \omega_n^{1/n}
\Vert f\Vert_{\infty, n}$
\end{cor}

\begin{rmk} The result refines the inequalities (\ref{des}), (\ref{BMR}) and (\ref{hs}).
\end{rmk}

\begin{proof}
It follows from the previous Theorem and the facts stated in the introduction
$$
\left(1-\frac 1q\right)n\ \omega_n^{1/n}\Vert f\Vert_{\infty, p}\le \mathcal{E}^{+}_p(f^{\circ}) \le
\mathcal{E}_p(f^\circ) \le  \mathcal{E}_p(f)\le \Vert\nabla f\Vert_p$$
\end{proof}

We pass to the case $p>n$. As we said in the introduction we shall see how Theorem \ref{relation}
provides better estimates than inequality (\ref{hsp}).

\begin{proposition}\label{mejora} Let $p>n$,
$\displaystyle\frac{1}{q}=\frac{1}{p}-\frac{1}{n}$ and $f$ a compactly supported
$C^{1}$
function. Then,
\[\sup_{t>0}\{ \left( \Vert
f\Vert_\infty-f^\ast(t)\right)t^{1/q}\}
\leq \alpha_{n,p}\Vert
f\Vert_{\infty,p}
\]
for some $\alpha_{n,p}>0$ (independent of the support of $f$).
\end{proposition}

\begin{rmk} The proof gives
$\alpha_{p,n}=\left(\big(p (1-\frac1{q})\big)^{p'/p}+\frac{|q|}{p'}\right)^\frac{1}{p'}$.
\end{rmk}

\begin{proof}

For any $t>0$ we have
\begin{eqnarray*}
\Vert f\Vert_\infty\!\!\!\! &-&\!\!\!\!f^\ast(t)=f^{\ast\ast}(0)-f^{\ast\ast}(t)+f^{\ast\ast}(t)-f^\ast(t)=
\int_0^t\!-f^{\ast\ast\prime}(u)du+f^{\ast\ast}(t)-f^\ast(t)\cr
&=&\int_0^t\frac{f^{\ast\ast}(u)-f^\ast(u)}{u}du+f^{\ast\ast}(t)-f^\ast(t)\cr
&=&\int_0^t\frac{f^{\ast\ast}(u)-f^\ast(u)}{u}du\left(\frac{p'}{|q|}\right)^\frac{1}{p'}
\left(\frac{|q|}{p'}\right)^\frac{1}{p'}
+\frac{f^{\ast\ast}(t)-f^\ast(t)}{\left(p\big(1-\frac{1}{q}\big)\right)^\frac{1}{p}}
\left(p\big(1-\frac{1}{q}\big)\right)^\frac{1}{p}.
\end{eqnarray*}

By H\"{o}lder's inequality, the latter expression is bounded from above by
$$
\alpha_{p,n}\left(\left(\int_0^t\frac{f^{\ast\ast}(u)-f^\ast(u)}{u}du \right)^p\left(\frac{p'}{|q|}\right)^\frac{p}{p'} +\frac{\left(f^{\ast\ast}(t)-f^\ast(t)\right)^p}{p\big(1-\frac{1}{q}\big)}\right)^\frac{1}{p}.
$$

Now, on one hand, for any $s>t>0$ we have after integrating by parts
\begin{align*}
 f^{\ast\ast}(s)-f^\ast(s)&=\frac 1{s}\int_0^{s} u|f^{\ast\prime}(u)|du
\geq \frac 1{s}\int_0^t u|f^{\ast\prime}(u)|du=\frac{t}{s}
(f^{\ast\ast}(t)-f^\ast(t)).
\end{align*}

Hence
$$
\int_t^{\infty}(f^{\ast\ast}(s)-f^\ast(s))^ps^{-p/n}ds
\ge  t^{p}(f^{\ast\ast}(t)-f^\ast(t))^p\int_t^{\infty}\frac{ds}{s^{p+\frac{p}{n}}}
$$
and consequently
$$
\frac{\left(f^{\ast\ast}(t)-f^\ast(t)\right)^p}{p\big(1-\frac{1}{q}\big)}\leq t^{-\frac{p}{q}}\int_t^\infty(f^{\ast\ast}(s)-f^\ast(s))^ps^{-p/n}ds.
$$

On the other hand, by H\"{o}lder's inequality and since $p>n$,

$$
\int_0^t\frac{f^{\ast\ast}(s)-f^\ast(s)}{s}\le
\left(\int_0^t s^{(\frac1{n}-1)p'}ds\right)^{1/p'} \left(\int_0^t(f^{\ast\ast}(s)-f^\ast(s))^ps^{-p/n}ds\right)^{1/p}
$$

which implies
$$
\left(\int_0^t\frac{f^{\ast\ast}(s)-f^\ast(s)}{s}\right)^p\left(\frac{p'}{|q|}\right)^{\frac{p}{p'}}t^{\frac{p}{q}}\leq 
\int_0^t(f^{\ast\ast}(s)-f^\ast(s))^ps^{-p/n}ds.
$$

Thus, for any $t>0$ we have
$$t^\frac{1}{q}(\Vert f\Vert_\infty-f^\ast(t))\leq\alpha_{p,n}\Vert f\Vert_{\infty,p}$$
which finishes the proof.
\end{proof}

\begin{rmk}  $f^\ast(|{\text {supp}\,f}|_n)=0$ implies
$\displaystyle \Vert f\Vert_\infty |{\text {supp}\,
f}|_n^{1/q}\le\sup_{t>0}\{ \left( \Vert
f\Vert_\infty-f^\ast(t)\right)t^{1/q}\}$ and so Proposition \ref{mejora} shows that
Theorem \ref{relation} is, up to constant, better than (\ref{hsp}).  The example $f^\ast(t)=(1-t^{-1/q})\chi_{[0,1]}(t)$ verifies
$\displaystyle\sup_{t>0}\{ \left( \Vert f\Vert_\infty{-}f^\ast(t)\right)t^{1/q}\}=1$ while $\Vert
f\Vert_{\infty, p}=\infty$.
\end{rmk}

\section{A simplified approach to affine Sobolev inequalities}

In this part we will show a direct way to deduce the key inequalities
(\ref{LYZ}) and (\ref{HS}) from
Zhang's inequality, paying a penalty on the constant.
We use  similar ideas to those appearing in \cite{Mo} which prove
Polya-Szeg\"o rearrangement inequality from the isoperimetric inequality.

\begin{proposition} Let $1\leq p<\infty$. For all $f\in W^{1,p}(\R^n)$
\[
 \mathcal{E}_p(f^\circ)\leq \frac{I_p}{I_1}\mathcal{E}_p\,
(f)
\qquad {\rm and}\qquad
 \mathcal{E}_p^+(f^\circ)\leq \frac{I_p}{I_1}\,\mathcal{E}_p^+(f).
\]
\end{proposition}

\begin{rmk} It is well known that $C_1\sqrt p\le\frac{I_p}{I_1}\le C_2\sqrt p$  with $C_1, C_2$ absolute constants.

\end{rmk}

\begin{proof} Suppose $f$ is a ${\mathcal C}^1$ function of compact support. 
Let $\Phi(t)$ represent either $|t|$ or $\max\{t,0\}$. By Sard's theorem, for almost all $t>0$
the level set $\{|f|\ge t\}$ is compact with ${\mathcal C}^1$ boundary $\{|f|=t\}$ and 
$\nabla f(x)\ne 0, \forall x\in \{|f|=t\}$.
 By Federer's co-area formula 
\[\int_{\R^n}\Phi(\langle \nabla f(x), u\rangle)^pdx
=\int_0^\infty\left(\int_{\{|f|=t\}}\Phi(\langle \nabla f(x),
u\rangle)^pd\mu(x)\right)dt\]
 where, for almost every $t>0$, $
       d\mu(x)=\displaystyle\frac{dH_{n-1}(x)}{|\nabla f(x)|}
      $   being $dH_{n-1}(x)$ the corresponding Hausdorff measure on 
$\{|f|=t\}$. Next we use Jensen inequality
\begin{align*}
 \int_{\R^n}\!\!\!\Phi(\langle \nabla f(x), u\rangle)^pdx&\ge
\int_0^\infty\!\!\!\left(\int_{\{|f|=t\}}\!\!\!\!\!\!\Phi(\langle \nabla f(x),
u\rangle)\frac{d\mu(x)}{\int_{\{|f|=t\}}d\mu(x)} \right)^p
\left(\int_{\{|f|=t\}}\!\!\!\!\!\!\!d\mu(x)\right)dt\\
&= \int_0^\infty\left(\int_{\{|f|=t\}}\!\!\!\!\Phi(\langle \nabla f(x),
u\rangle) d\mu(x) \right)^p \left(\int_{\{|f|=t\}}\!\!\!\!d\mu(x)\right)^{1-p}dt.
\end{align*}
Denote $\displaystyle M(t)= \int_{\{|f|\ge t\}}\!\!\!\!dx$. For
 almost every $t>0$, another use of the co-area formula yields to
\[
 \int_{\{|f|=t\}}d\mu(x)=\left(-\int_t^\infty
\left(\int_{\{|f|=s\}}\frac{dH_{n-1}(x)}{|\nabla
f(x)|}\right)ds\right)^\prime=-M^\prime(t)=|M^\prime(t)|
\]

and so
\begin{align*}
& \left(\int_{S^{n-1}}\left(\int_{\R^n}\Phi(\langle \nabla f(x),
u\rangle)^pdx\right)^{-n/p}
du\right)^{-p/n}\\
&\geq
\left(\int_{S^{n-1}}\left(\int_0^\infty\left(\int_{\{|f|=t\}}\Phi(\langle \nabla
f(x),
u\rangle) d\mu(x) \right)^p |M^\prime(t)|^{1-p}dt
\right)^{-n/p}du\right)^{-p/n}\!\!\!.\end{align*}

Use Minkowski's integral inequality to bound the previous formula from below
\begin{align*}
&\geq \int_0^\infty \left(\int_{S^{n-1}}\left(\int_{\{|f|=t\}}\Phi(\langle
\nabla f(x),
u\rangle) d\mu(x) \right)^{-n}du
\right)^ {-p/n}|M^\prime(t)|^{1-p}dt\\
& =
\int_0^\infty \left(\int_{S^{n-1}}\left(\int_{\{|f|=t\}}\Phi(\langle \nu(x),
u\rangle)dH_{n-1}(x)\right)^{-n}du
\right)^ {-p/n}|M^\prime(t)|^{1-p}dt\end{align*}
where for a.e. $t>0$, $\nu(x)$ is the outer normal vector to $\{|f|=t\}$ (w.r.t. $\{|f|\ge t\}$) at the point $x$.

For every ``good'' $t>0$ from Sard's theorem, one can easily see, \cite{Z}, that the linear functional
 $g\in\mathcal{C}(S^{n-1})\to \int_{\{|f|=t\}} g(\nu(x))dH_{n-1}(x)$ can be represented by a 
finite measure $d\mu_t$ on $S^{n-1}$. That is 
\[\int_{S^{n-1}}g(v)d\mu_t(v)=\int_{\{|f|=t\}} g(\nu(x))dH_{n-1}(x)\hskip 1truecm
\forall\, g\in\mathcal{C}(S^{n-1}).\]

Recall (see for instance \cite{S}) that every convex body $K\subset\R^n$ determines a 
{\sl surface area measure} on $S^{n-1}$ denoted by $S_K$.

It is also proved in \cite{Z} that, by Minkowski existence theorem (see \cite{S}), there exists a unique up to
translations convex body $K_t$ in $\R^n$ whose surface area measure $S_{K_t}$ is $\mu_t$. 
For this reason $\mu_t$ is also called the surface area measure of (the compact set) $\{|f|\ge t\}$. 

Let $\Pi K_t$ be the projection body associated to $K_t$, i.e. the convex body defined by its
support function as
\[
 h(\Pi K_t,u)=\frac12\int_{S^{n-1}}|\langle v,
u\rangle|dS_{K_t}(v)=|P_{u^\perp}(K_t)|_{n-1}, \qquad u\in S^{n-1}.
\]
 Let $\Pi^\ast
K_t$ be the polar projection body of $K_t$. Its volume is 
\begin{align*} 
 2^{-n}|\Pi^\ast K_t|_n & =\int_{S^{n-1}}h(\Pi K_t,u)^{-n}du=
\int_{S^{n-1}}\left(\int_{S^{n-1}}|\langle
u,
v\rangle|dS_{K_t}(v)\right)^{-n}du\\ &= 
\int_{S^{n-1}}\left(\int_{\{|f|=t\}}|\langle
\nu(x),
u\rangle|dH_{n-1}(x)\right)^{-n}du.
\end{align*}

Finally, Petty's projection inequality (\ref{Petty}) and the fact proved in \cite{Z}, $M(t)=|\{|f|\geq t\}|_n\leq  |K_t|_n$, 
show that for almost all $t>0$
\begin{align*}
 \left(\int_{S^{n-1}}\left(\int_{\{|f|=t\}}|\langle \nu(x),
u\rangle|dH_{n-1}(x) \right)^{-n}du\right)^ {-1/n}\geq
2\frac{\omega_{n-1}}{\omega_n^{1-1/n}}
 M(t)^{\frac{(n-1)}{n}}.
\end{align*}

Consider the case $\Phi(t)=|t|$.  $M$ and $f^{\ast}$ are differentiable (except, possibly, on some set $N$ of zero measure) 
and $M=(f^{\ast})^{-1}$ on the set $(f^{\ast})^{-1}\big((0,\infty)\setminus N\big)$, therefore
\begin{align*}
\left(\int_{S^{n-1}}\Vert D_uf\Vert_p^{-n}du\right)^{-p/n} &\geq
2^{p}\left(\frac{\omega_{n-1}}{\omega_n^{1-1/n}}\right)^p
\int_0^\infty M(t)^{\frac{(n-1)p}{n}}|M^\prime(t)|^{1-p}dt\\
&=\left(\frac{2\omega_{n-1}}{\omega_n^{1-1/n}}\right)^p\int_0^\infty
s^{\frac{p(n-1)}{n}}|f^{\ast\prime}(s)|^pds.
\end{align*}

In the other case, $\Phi(t)=\max\{ t,0\}$, since $\int_{S^{n-1}}\langle
u,v\rangle d\mu_t(v)=0$ we have
$$ h(\Pi K_t,u)=\frac12\int_{S^{n-1}}|\langle v,
u\rangle|d\mu_t(v)
=\int_{S^{n-1}}\langle v,
u\rangle_+d\mu_t(v).
$$
Hence
\begin{align*}
\left(\int_{S^{n-1}}\Vert D_u^+f\Vert_p^{-n}du\right)^{-p/n} &\geq
\left(\frac{\omega_{n-1}}{\omega_n^{1-1/n}}\right)^p\int_0^\infty
s^{\frac{p(n-1)}{n}}|f^{\ast\prime}(s)|^pds.
\end{align*}

Finally we recall (\ref{fpolar}) and analogously
\[
 \left(\int_{S^{n-1}}\Vert D_uf^\circ\Vert_p^{-n}du\right)^{-p/n}=
\Big(n\omega_n^{1/n}\Big)^p\left(\int_0^\infty
s^{\frac{(n-1)p}{n}}|f^{\ast\prime}(s)|^pds
\right)I_p^p.
\]
Therefore
\[\mathcal{E}_p(f^\circ)\leq \frac{n\omega_n}{2\omega_{n-1}}
I_p\mathcal{E}_p(f)=\frac{I_p}{I_1}\mathcal{E}_p(f)\quad \text{and}\quad \mathcal{E}_p^+(f^\circ)\leq \frac{n\omega_n}{2\omega_{n-1}}
I_p\mathcal{E}_p^+(f)=\frac{I_p}{I_1}\mathcal{E}_p^+(f).\]

\end{proof}

{\bf Acknowledgements.} We thank the referee for useful 
comments that helped to improve the presentation of the manuscript.

\end{document}